\documentclass[12pt]{amsart}
\usepackage{amsmath,amssymb,amsbsy,amsfonts,latexsym,amsopn,amstext,
                                               amsxtra,euscript,amscd}
\usepackage{url}
\usepackage[colorlinks,linkcolor=blue,anchorcolor=blue,citecolor=blue]{hyperref}
\usepackage{color}
\usepackage{graphics,epsfig}
\usepackage{graphicx}
\usepackage{float}

\usepackage[english]{babel}
\begin{document}

\newtheorem{theorem}{Theorem}
\newtheorem{lemma}[theorem]{Lemma}
\newtheorem{claim}[theorem]{Claim}
\newtheorem{cor}[theorem]{Corollary}
\newtheorem{prop}[theorem]{Proposition}
\newtheorem{definition}{Definition}
\newtheorem{quest}[theorem]{Open Question}

\numberwithin{equation}{section}
\numberwithin{theorem}{section}

 \newcommand{\F}{\mathbb{F}}
\newcommand{\K}{\mathbb{K}}
\newcommand{\D}[1]{D\(#1\)}
\def\scr{\scriptstyle}
\def\\{\cr}
\def\({\left(}
\def\){\right)}
\def\[{\left[}
\def\]{\right]}
\def\<{\langle}
\def\>{\rangle}
\def\fl#1{\left\lfloor#1\right\rfloor}
\def\rf#1{\left\lceil#1\right\rceil}
\def\le{\leqslant}
\def\ge{\geqslant}
\def\eps{\varepsilon}
\def\mand{\qquad\mbox{and}\qquad}
\def\ep{\mathbf{e}_p}
\def\e{\mathbf{e}}
\def\vec#1{\mathbf{#1}}

\newcommand{\commD}[1]{\marginpar{%
\begin{color}{red}
\vskip-\baselineskip %raise the marginpar a bit
\raggedright\footnotesize
\itshape\hrule \smallskip D: #1\par\smallskip\hrule\end{color}}}

\newcommand{\commI}[1]{\marginpar{%
\begin{color}{blue}
\vskip-\baselineskip %raise the marginpar a bit
\raggedright\footnotesize
\itshape\hrule \smallskip I: #1\par\smallskip\hrule\end{color}}}

\newcommand{\Fq}{\mathbb{F}_q}
\newcommand{\Q}{\mathbb{Q}}
\newcommand{\C}{\mathbb{C}}
\newcommand{\cS}{\mathcal{S}}
\newcommand{\Fp}{\mathbb{F}_p}
\newcommand{\Z}{\mathbb{Z}}
\newcommand{\PP}{\mathbb{P}}
\newcommand{\Disc}[1]{\mathrm{Disc}\(#1\)}
\newcommand{\Res}[1]{\mathrm{Res}\(#1\)}
%\newcommand{\Nm}[1]{\mathrm{Norm}_{\F{q^k/\Fq}}(#1)}

%%%%%%%%%%%%%%%%%%%%%%%%%
% Alphabet calligraphic %
%%%%%%%%%%%%%%%%%%%%%%%%%
\def\cA{{\mathcal A}}
\def\cB{{\mathcal B}}
\def\cC{{\mathcal C}}
\def\cD{{\mathcal D}}
\def\cE{{\mathcal E}}
\def\cF{{\mathcal F}}
\def\cG{{\mathcal G}}
\def\cH{{\mathcal H}}
\def\cI{{\mathcal I}}
\def\cJ{{\mathcal J}}
\def\cK{{\mathcal K}}
\def\cL{{\mathcal L}}
\def\cM{{\mathcal M}}
\def\cN{{\mathcal N}}
\def\cO{{\mathcal O}}
\def\cP{{\mathcal P}}
\def\cQ{{\mathcal Q}}
\def\cR{{\mathcal R}}
\def\cS{{\mathcal S}}
\def\cT{{\mathcal T}}
\def\cU{{\mathcal U}}
\def\cV{{\mathcal V}}
\def\cW{{\mathcal W}}
\def\cX{{\mathcal X}}
\def\cY{{\mathcal Y}}
\def\cZ{{\mathcal Z}}%%

\def\Inq{\cI_{n,q}}
\def\Inp{\cI_{n,p}}
\def\Pnq{\cP_{n,q}}
\def\Pnp{\cP_{n,p}}
\def\Qnq{\cP_{n,q}^*}
\def\Qnp{\cP_{n,p}^*}
\def\Snq{S_{n,q}(\psi)}
\def\Tnq{S_{n,q}^*(\psi)}
\def\Tnpa{T_{n,p}^*(a)}
\def\Tnpv{T_{n,p}^*(v)}
\def\Vnq{V_{n,q}(\psi)}

\newcommand{\Nm}[1]{\mathrm{Norm}_{\,\F_{q^k}/\Fq}(#1)}

\def\Tr{\mbox{Tr}}
\newcommand{\rad}[1]{\mathrm{rad}(#1)}

\title[Chains and Reflections
of $g$-ary Expansions]{Arithmetic Properties of Integers in Chains and Reflections
of $g$-ary Expansions}
%\maketitle

%
\author[D. G\'omez-P\'erez] {Domingo G\'omez-P\'erez}
\address{Department of  Mathematics, Statistics and Computation,
Universidad de Cantabria,
Santander, 39005 Cantabria, Spain}
\email{domingo.gomez@unican.es}

\author[I. E. Shparlinski]{Igor E. Shparlinski}
\address{Department of Pure Mathematics, University of New South Wales,
Sydney, NSW 2052, Australia}
\email{igor.shparlinski@unsw.edu.au}

\begin{abstract}
Recently, there
has been a sharp rise of interest in properties of digits primes.
Here we study yet another question of this kind.
Namely,  we fix an  integer base $g \ge 2$ and then for every infinite  sequence
$$\cD = \{d_i\}_{i=0}^\infty \in \{0, \ldots, g-1\}^\infty
$$
of $g$-ary digits
we consider the counting function $\varpi_{\cD,g}(N)$ of integers $n \le N$
for which $\sum_{i=0}^{n-1} d_i g^i$ is prime. We construct sequences $\cD$
for which  $\varpi_{\cD,g}(N)$ grows fast enough, and show that for some
constant $\vartheta_g< g$ there are at most $O(\vartheta_g^N)$   initial elements
 $(d_0, \ldots, d_{N-1})$ of $\cD$ for which  $\varpi_{\cD,g}(N)=N+O(1)$.
We also discuss  joint  arithmetic properties of integers and mirror
reflections of their $g$-ary expansions. 
\end{abstract}

\keywords{primes, digits}
\subjclass{11A41, 11A63, 11N05}
\maketitle

\section{Introduction}

%% \subsection{Motivation}

We fix an integer base $g \ge 2$ and then
for every infinite  sequence
$$\cD = \{d_i\}_{i=0}^\infty \in \{0, \ldots, g-1\}^\infty
$$
of $g$-ary digits.

We say that $\cD$ is of {\it length $N$\/} if $d_{N-1}$ is
the last non-zero element of $\cD$ if such $N$ exists;
otherwise we say that $\cD$ is of {\it infinite  length\/}.

We form the sequence  of integers
\begin{equation}
\label{eq:chain}
u_{\cD,g}(n) = \sum_{i=0}^{n-1} d_i g^i
\end{equation}
and define the counting function $\varpi_{\cD,g}(N)$ of integers
$n \le N$, for which
$u_{\cD,g}(n)$ is prime.
The question of the distribution of prime values in the sequences~\eqref{eq:chain}
has been introduced by Angell and  Godwin~\cite{AnGo},
see also~\cite{vdP}.
More precisely, both papers~\cite{AnGo,vdP} study sequences $\cD$ such that
the elements of~\eqref{eq:chain} are all primes.
Analogues of this question for polynomials
over finite fields have been considered by Chou and Cohen~\cite{ChCo}
and more recently by G{\'o}mez-P{\'e}rez,  Ostafe and Sha~\cite{GOS},
which  have in fact motivated this work. The rest of our motivation comes 
from a series of recent striking results about primes with restricted digits~\cite{Bour,DMR,MauRiv,May}. 

It is easy to see that for almost all sequences  $\cD$
 (in the sense of the Lebesgue measure
in one interpret $\cD$ as  a $g$-ary expansion of a real number in $[0,1]$)
we have $u_{\cD,g}(n)  = g^{n + o(n)}$ (and in fact $g^{n-1} \le u_{\cD,g}(n)  < g^n$
if $\cD$ does not contain zero digits). Hence, the standard heuristic suggests that
 $\varpi_{\cD,g}(N)$  has to grow as
 $$
 \sum_{n=1}^{N} \frac{1}{\log u_{\cD,g}(n)} =  \sum_{n=1}^{N} \frac{1}{n \log g+ o(n)}
 \sim  \frac{\log N}{ \log g}
 $$
(clearly there are also some local conditions which we have ignored as we are only
interested in the rate of growth).

On the other hand, one can clearly guarantee that $\varpi_{\cD,g}(N) \ge  N-1$
by simply taking  $d_0=0$, $d_1 =1$ if $g = 2$ and
$d_0=2$, if $g \ge 3$, setting all other elements to zero.
However, we are interested in prime values of $u_{\cD,g}(n)$
for nontrivial sequences $\cD$ of large or infinite length.

 For instance,  in Section~\ref{sec:const} we construct a sequence $\cD \in \{0, \ldots, g-1\}^\infty$
with infinitely  many non-zero digits for which
\begin{equation}
\label{eq:LowBound}
\varpi_{\cD,g}(N) \gg  \log N
\end{equation}
where, as usual, the expressions $A \ll B$,  $B \gg A$ and $A=O(B)$ are each equivalent to the
statement that $|A|\le cB$ for some positive constant $c$.
Throughout the paper the implied constants may   depend on $g$.

Let  $P_g(N)$  be the number of sequences  $\cD$
of length $N$ (that is, with $d_{N-1} \ne 0$)
such that $\varpi_{\cD,g}(N) =  N-  \eta$,
where
$$
\eta = \left\{
\begin{array}{ll} 1, & \text{if}\ g = 2,\\
0, & \text{if}\ g \ge 3.
\end{array}
\right.
$$
In particular, from the prime number theorem we immediately
obtain the following trivial bound
$$
P_g(N) \ll \frac{g^N}{N},
$$
which we use as a benchmark for our improvements in
 Section~\ref{sec:bound}.

We also use this opportunity to introduce another question about digits of primes. 
Namely, given a  $g$-ary expansion
$$
s = \sum_{i=0}^{n-1}  d_i g^{i}, \qquad d_i \in \{0,\ldots, g-1\}, \ i =1, \ldots, n,
$$
we denote by $s_g^*$ the ``mirror'' reflection of $s$, that is,
$$
s_g^* = \sum_{i=0}^n d_{n-1-i} g^{i}.
$$
We denote by $M_g(N)$ the number of primes $p \in [g^{N-1}, g^N-1]$,
for which $p_g^*$ is also prime. For example, if $p$ is a Fermat prime, then
$p_{2}^*$ is also a prime.
Although we have not been able to obtain any nontrivial bounds on $M_g(N)$,
in Section~\ref{sec:mirror} we give some
other results about the simultaneous arithmetic structure of $p$ and $p_g^*$.
In passing, we note that corresponding question for polynomials is trivial
as the ``mirror'' polynomial $X^Nf(1/X)$ of a polynomial $f(X)$
of degree $N$ has the same arithmetic structure as $f$.

\section{Constructing Sequences With Many Primes}
\label{sec:const}

Using the bound of  Chang~\cite[Corollary~11]{Chang} on the smallest 
prime in an arithemtic progression modulo an integer composed out
of small primes (see also~\cite{Iwan}), 
we obtain the following more precise form of~\eqref{eq:LowBound}.

\begin{theorem}
\label{thm:InfSeq}
There is a sequence
$$\cD = \{d_i\}_{i=0}^\infty \in \{0, \ldots, g-1\}^\infty
$$
that has infinitely many non-zero elements,  for which
$$
\varpi_{\cD,g}(N)  \ge \(\frac{1}{\log (12/5)} + o(1)\)\log N
$$
as $N \to \infty$.
\end{theorem}

\begin{proof} We choose $d_0$ and $d_1$ in such a way that either
$u_{\cD,g}(1)$ or $u_{\cD,g}(2)$
is prime. Now, assume that $d_0, \ldots, d_{n-1}$ have already been chosen
for the initial segment of
$\cD$.  Using a result of Chang~\cite[Corollary~11]{Chang} (see also~\cite{Iwan})
we see that for any $\varepsilon > 0$ there exists a
constant $c$ such that for every $n =1,2, \ldots$
there exists a prime $p$ with
$$
p \equiv u_{\cD,g}(n) \pmod {g^n} \mand p \le c g^{n(12/5+\varepsilon)}
$$
(note that $\gcd(u_{\cD,g}(n), g)=1$), see also for much larger class of moduli than $g^n$.
We now define the next
$$
m=  (7/5 + \varepsilon)n + O(1)
$$
elements of $\cD$ as the
$g$-ary digits of $\(p-u_{\cD,g}(n)\)/g^n$. This implies the inductive inequality
$\varpi_{\cD,g}(n+m) \ge \varpi_{\cD,g}(n) + 1$.
Thus for $N_k = \rf{(12/5 +  2\varepsilon)^k}$, $k =1,2, \ldots$,
we obtain $\varpi_{\cD,g}(N_k) \ge k+ O(1)$.
 Since $\varepsilon$ is arbitrary,
the result now follows.
\end{proof}

\section{Bounding the number of sequences with  all  primes}
 \label{sec:bound}

We now use a version of the Brun--Titchmarsh inequality,
due to  Montgomery and Vaughan~\cite[Theorem~2]{MoVa}   to improve the trivial
upper bound~\eqref{eq:LowBound}.

For $g \ge 3$ we define
\begin{equation}
\label{eq:gamma}
\gamma_g =  g \min_{m =1,2, \ldots}  \(\frac{2g}{m\varphi(g) \log g}\)^{1/m},
\end{equation}
where $\varphi(q)$ is the Euler function of the integer $q \ge 1$.

Clearly $\gamma_g < g$ for any $g$ and also when $g$ is large enough then
$m=1$ is the optimal value and thus
$$
\gamma_g =\frac{2g^2}{\varphi(g) \log g}.
$$
  The standard bound on the Euler function
(see~\cite[Theorem~328]{HardyWright}) guarantees
that
$$
\gamma_g =O\(\frac{g \log \log g}{\log g}\)
$$
as $g \to \infty$.

\begin{theorem}
\label{thm:PrimeSeq  Any g}
For a sufficiently large $N$, we have
$$
P_g(N)  \ll \gamma_g^N, 
$$
where the implied constant is absolute.
\end{theorem}

\begin{proof} We start with deriving an inductive inequality
between $P_g(n)$ and $P_g(n+m)$ for an appropriately chosen $m$.

We first observe that the first  $n$ digits of any $n+m$ digit  sequence
$(d_0, \ldots, d_{n+m-1})$ counted in  $P_g(n+m)$
must come from a sequence counted in  $P_g(n)$.
Now,  assume that $m > \eta$. Then, all
such  extensions of  a  $n$ digit sequence  
to a $n+m$ digit  sequence counted in  $P_g(n+m)$
generates a prime  $p \le g^{n+m}$
in a fixed arithmetic progression modulo $g^n$.
We now recall the upper bound from~\cite[Theorem~2]{MoVa}
\begin{equation}
\label{eq:BrunTitchmarsh}
\pi(x;q,a) \le \frac{2 x}{\varphi(q) \log (x/q)},
\end{equation}
on the
number of primes $p\le x$ in arithmetic progressions $p \equiv a \pmod q$,
(see also~\cite[Theorem~6.6]{IwKow}
for a slightly weaker result, which is still sufficient for our purposes).
Therefore, we obtain
$$
P_g(n+m) \le P_g(n)  \frac{2 g^{n+m}}{\varphi(g^n) \log g^m}
= P_g(n) \frac{2g^{m+1}}{m\varphi(g) \log g}.
$$
We now conclude that for any fixed integer $m\ge 1$, denoting by $r$ the remainder of $N$ on
division by $m$, and using the trivial bound $P_g(r) \le g^r$, we have
\begin{align*}
P_g(N)& \le  g^r \(\frac{2g^{m+1}}{m\varphi(g) \log g}\)^{\fl{N/m}}
=g^N \(\frac{2g}{m\varphi(g) \log g}\)^{\fl{N/m}} \\
& \ll g^N \(\frac{2g}{m\varphi(g) \log g}\)^{N/m}.
\end{align*}
with  an absolute implied constant.
Simple calculus shows that
$$
\lim_{z\to \infty}  \(\frac{2g}{z \varphi(g) \log g}\)^{1/z} = 1.
$$
Hence  there is  integer  $m_0$,  depending only on $g$ , on which the minimum in~\eqref{eq:gamma}
is achieved, and the result now follows.
\end{proof}

We note that for the values of $q$ in the medium range, for example, for $x^\vartheta \le q \le 2x^{\vartheta}$
for some fixed real $\vartheta \in (0,1)$, there are various improvements of~\eqref{eq:BrunTitchmarsh},
see~\cite{BouGar,FrIw} and references therein. However these results do not seem to
be useful in our context.

On the other hand, for smaller values of $g$ one can obtain better values of $\gamma_g$ via an
application of  the sieve of Eratosthenes instead of a direct application
of~\cite[Theorem~2]{MoVa} (in fact implicitly this is a part of the argument of
the proof of~\cite[Theorem~2]{MoVa}, see~\cite[Lemma~3]{MoVa}).

For positive integers $q$ and $U$ we define the function
$$
\varphi(q,U) =
\max_{1\le h \le q}\,  \sum_{\substack{u=1 \\ \gcd(u+h, q)=1}}^U 1 .
$$
In particular,  $\varphi(q,q) = \varphi(q)$ is the classical Euler function.
We also note that  it can be defined in a more general but equivalent
form
$$
\varphi(q,U) =
\max_{\substack{1 \le a,b \le q\\ \gcd(a,q) =1}}\,  \sum_{\substack{u=1 \\ \gcd(au+b, q)=1}}^U 1.
$$
Using the M\"obius function $\mu(d)$ over the divisors of
$q$ to detect
the co-primality condition, see~\cite[Equation~(1.18)]{IwKow} and interchanging the order of summation, we
derive
$$
 \sum_{\substack{u=1 \\ \gcd(u+h, q)=1}}^U 1 = \sum_{d\mid q}\mu(d)
\sum_{\substack{u=1 \\  d \mid u+h}}^U 1 =
\sum_{d\mid q}\mu(d) \(\frac{U}{d} + \xi_d\),
$$
where $|\xi_d| \le (d-1)/d$ (since the condition $d \mid u+h$ puts $u$ in a
prescribed arithmetic progression modulo $d$).
Hence, using $s$ to denote the number of prime divisors of $q$ we obtain
\begin{equation}
\label{eq:Erat}
\begin{split}
 \sum_{\substack{u=1 \\ \gcd(u+h, q)=1}}^U 1  \le
\sum_{d\mid q}\mu(d) \frac{U-1}{d}  + \sum_{d\mid q}|\mu(d)|
=
  \frac{\varphi(q)}{q}(U-1) + 2^s
\end{split}
\end{equation}
by~\cite[Equation~(1.36)]{IwKow}.
However for our purposes below, we work with rather small values
of $q$ and $U$, so we can always compute $\varphi(q,U)$ explicitly.

We can now use the above argument to improve  the values of $\gamma_g$ of
Theorem~\ref{thm:PrimeSeq Any g}
for $g =2, 3, 5, 10$,
and show that
$$
\gamma_2= 1.876\ldots, \quad
\gamma_3 =  2.622\ldots, \quad
\gamma_5 =  3.947\ldots, \quad
\gamma_{10} =  8.441\ldots
$$
(corresponding to the $m= 16,7, 4, 6$, respectively, in~\eqref{eq:gamma}).
However, instead of using the bound~\eqref{eq:Erat} directly, we simply
evaluate $\varphi(q,U)$ for concrete values of $q$ and $U$ that optimize
our results.

\begin{theorem}
\label{thm:PrimeSeq Small g}  For  $g =2, 3, 5, 10$, we have,
$$
P_g(N)  \ll \vartheta_g^N,
$$
where
$$
\vartheta_2 = 5^{1/3} = 1.709\ldots,
%%  2\cdot \(\frac{289}{560}\)^{1/7} = 1.8196\ldots
 %% 5284489049\
 \qquad  \vartheta_3 =  2,
  \qquad  \vartheta_5 =  2,
   \qquad  \vartheta_{10} =  6.
   %% 6\cdot (15)^{-1/3}=2.4328\ldots
 %% 8079822936
 $$
\end{theorem}

\begin{proof} We present the argument in
a rather generic form suitable for further generalizations.
Let $s$ be an appropriately chosen integer and  let $q_s$ be a product 
of first $s$  primes that are relatively prime to $g$.

We proceed inductively as in the proof of Theorem~\ref{thm:PrimeSeq Any g}.
We assume that   $n$ is large enough so that we always have $u_{\cD,g}(n) > q_s$.
However, now,  instead of requesting that the extended sequence
$\cD_{n+m}= (d_0, \ldots, d_{n+m-1})$ corresponds to prime values of $u_{\cD,g}(n+m)$
we merely request that $\gcd(u_{\cD,g}(n+m), q_s)=1$.
Hence,
%% by Lemma~\ref{lem:erat}
%:
$$
P_g(n+m) \le P_g(n) \varphi(q_s, g^m).
%% \((g^m-1)  \frac{\varphi(q_s)}{q_s} + 2^s\).
$$
A simple inductive argument implies that for any fixed $m$ we have
$$
P_g(N)
\ll g^N \(g^{-m} \varphi(q_s, g^m)\)^{N/m}
=  \varphi(q_s, g^m)^{N/m}.
%%\( (1-g^{-m}) \frac{\varphi(q_s)}{q_s} + 2^{s} g^{-m}\)^{N/m}.
$$
Now, for $g=2$ we take $s = 2$ (so $q_s= 15$) and $m = 3$.
For $g=3$ we take  $s = 1$ (so $q_s=2$) and $m = 1$.
For $g=5$ we take  $s = 2$ (so $q_s=6$) and $m = 1$.
Finally, for $g=10$ we take  $s = 2$ (so $q_s=21$) and $m = 1$.
\end{proof}

We remark that  it is
quite possible that the elementary method of the proof of Theorem~\ref{thm:PrimeSeq Small g}
always improves on Theorem~\ref{thm:PrimeSeq Any g}, but it seems to be more difficult to
analyze.

One can also obtain similar results for sequences generating square free
integers. More precisely,  we can proceed exactly as in Theorem~\ref{thm:PrimeSeq Small g}  but instead
count integers in short intervals which fall in  residues classes $a \pmod {q_s^2}$, where $\gcd(a,q_s^2)$
is square free.

Since we allow zero digits, one expects that $P_g(N)$ is a growing
function of $N$ as for any prime ``ending'' $p$ one  expects to find $N$ such that
$g^N + p$ is prime again.

This expectation is based on the standard heuristic predicting
primes in increasing sequences of integers (without any local obstructions). Namely,
since the series
$$
\sum_{N=1}^{\infty} \frac{1}{\log (g^N + p)} = \infty
$$
is diverging, for any  $p$ with $\gcd(g,p)=1$ there are probably infinitely many positive
integers $N$ for which $g^N + p$ is prime (we need only one such $N$).
Our numerical tests suggest that in fact $P_g(N)$ grows exponentially, 
see Figure~\ref{fig:PnM}.

\section{Prime Mirrors in Arithmetic Progressions}
 \label{sec:mirror}

For positive integers $N$ and $m$ and an arbitrary integer $a$,
we denote by $R_g(N,m,a)$ the number of primes
$p \in [g^{N-1}, g^N-1]$
such that
$$p_g^* \equiv a \pmod m.
$$

For $R_g(N,m,a)$ we have the following two trivial bounds
\begin{equation}
\label{eq:R triv}
R_g(N,m,a) \ll \frac{g^N}{N} \mand R_g(N,m,a) \ll \frac{g^N}{m}+1.
\end{equation}

We now obtain a bound which improves~\eqref{eq:R triv} in the medium
range.
\begin{theorem} \label{thm: RNma}
For any   integer $m\ge 1$ we have
$$
R_g(N,m,a) \ll \frac{g^N\gcd(g^N,m)}{N m^{1/2}} .
%% + 1.
$$
\end{theorem}

\begin{proof}  We choose some integer parameter $r\ge 1$ and consider the
integers
$$
1\le b_1 <  \ldots < b_t\le  g^r-1,
$$
formed by  the  top $r\le N$ $g$-ary digits of primes $p \in [g^{N-1}, g^N-1]$.
Clearly, for at least $(t-1)/2$
values of $i = 1, \ldots,t-1$, we have
$$
b_{i+1} - b_i \le H,
$$
where
$$
H = 2 \frac{g^{r}}{t-1}.
$$

Let $Q(N,h)$ be the number of primes $p \le g^N$ such that
$p+h$ is also prime.
Then we see that
\begin{equation}
\label{eq:t Q}
\frac{t-1}{2} \le \sum_{1 \le h \le H} Q(N,h).
\end{equation}

Using a very special case of the classical result of
Halberstam and  Richert~\cite[Theorem~3.12]{HR},
we see that
\begin{equation}
\label{eq:Q sieve}
Q(N,h)\ll \frac{g^N}{N^2} \prod_{\substack{\ell\mid h,~\ell\nmid g\\ \ell\ge 3~\text{prime}}}
\(\frac{\ell-1}{\ell-2}\).
\end{equation}
It is easy to show that
\begin{equation}
\label{eq:av val}
\sum_{1 \le h \le H} \prod_{\substack{\ell\mid h\\ \ell~\text{odd prime}}}
\(\frac{\ell-1}{\ell-2}\) \ll H.
\end{equation}
For example, using the elementary
inequality
$$
 \frac{z-1}{z-2} \le   \(\frac{z}{z-1}\)^2
 $$
 that holds for $z \ge 3$, we
  obtain
  $$
\prod_{\substack{\ell\mid h,~\ell\nmid g\\ \ell\ge 3~\text{prime}}}
\(\frac{\ell-1}{\ell-2}\) \le
\prod_{\substack{\ell\mid h,~\ell\nmid g\\ \ell\ge 3~\text{prime}}}
\(\frac{\ell}{\ell-1}\)^2  \le \( \frac{h}{\varphi(h)}\)^2
$$
and~\eqref{eq:av val} follows immediately from the general results of
Balakrishnan and   P{\'e}termann~\cite{BalPet}.

Thus, assuming that $t \ge 2$, substituting~\eqref{eq:Q sieve} and~\eqref{eq:av val}
in~\eqref{eq:t Q} we obtain
$$
t\ll H \frac{g^N}{N^2} + 1 \ll   \frac{g^{N+r}}{N^2 t} +1.
$$
Hence
\begin{equation}
\label{eq:t N r}
t\ll  \frac{g^{(N+r)/2}}{N}.
\end{equation}
Now, writing a prime $p$ as $p =g^{N-r}v + u$, with integers $u\in [0, g^{N-r}-1]$
and $v \in [g^{r-1}, g^r -1]$, we see that
$p_g^* = g^{r}u_g^* + v_g^*$.
Clearly, if we define $r$ by the condition
\begin{equation}
\label{eq:Cond r}
g^{N-r}\le  m < g^{N-r+1}
\end{equation}
%%then for each value of $v$, and thus of $v_g^*$, the congruence
%%$g^{r}u_g^* + v_g^* \equiv a  \pmod m$ defines $u_g^*$ and thus $u$
%%uniquely.
then for each value of $v$, and thus of $v_g^*$, the congruence
$$
g^{r}u_g^* + v_g^* \equiv a  \pmod m
$$
defines $u_g^*$ and thus $u$
in at most $\gcd(m,g^r) \le \gcd(m,g^N)$ ways.
Hence
$$
R_g(N,m,a) \le t \gcd(m,g^N),
$$
which together~\eqref{eq:t N r} and~\eqref{eq:Cond r} implies the desired bound.
\end{proof}

We now give an arithmetic application of  Theorem~\ref{thm: RNma}. In particular,
we show that the sum of divisors function
$$
\sigma(k) = \sum_{d \mid k} d
$$
grows linearly on average over the sequence $p_g^*$ for primes $p$ in
the interval $p \in [g^{N-1}, g^N-1]$.

\begin{cor} \label{cor: SumDiv}
We have,
$$
\frac{g^{N}}{N} \ll \sum_{p \in [g^{N-1}, g^N-1]} \frac{\sigma(p_g^*)}{p_g^*} \ll   \frac{g^{N}}{N}.
$$
\end{cor}

\begin{proof}
We set  $R_g(N,m) = R_g(N,m,0)$. We also write
\begin{equation*}
\begin{split}
\sum_{p \in [g^{N-1}, g^N-1]} \frac{\sigma(p_g^*)}{p_g^*}  &
=  \sum_{p \in [g^{N-1}, g^N-1]}  \sum_{ d \mid p_g^*}   \frac{d}{p_g^*}
=    \sum_{p \in [g^{N-1}, g^N-1]} \sum_{ d \mid p_g^*}   \frac{1}{d}\\
&=  \sum_{d < g^N}    \sum_{\substack{p \in [g^{N-1}, g^N-1] \\  p_g^* \equiv 0 \pmod d}} \frac{1}{d}
= \sum_{d < g^N}     \frac{1}{d} R_g(N,d).
 \end{split}
\end{equation*}
Using  Theorem~\ref{thm: RNma} we obtain
\begin{equation*}
\begin{split}
\sum_{p \in [g^{N-1}, g^N-1]} \frac{\sigma(p_g^*)}{p_g^*} & \ll  \frac{g^{N}}{N} \sum_{d < g^N}     \frac{\gcd(g^N,d)}{d^{3/2} }  \\
 & \ll   \frac{g^{N}}{N} \sum_{r=0}^N g^r \sum_{e < g^{N-r}}     \frac{1}{(g^re)^{3/2} }   \ll    \frac{g^{N}}{N}.
 \end{split}
\end{equation*}
The lower bound follows from the prime number theorem (see~\cite[Corollary~5.29]{IwKow}),
 as we trivially have  $\sigma(n)/n \ge 1$.
\end{proof}

We now denote by $\omega(k)$ the number of distinct prime divisors of an
integer $k \ge 1$.

\begin{cor} \label{cor: PrimeDiv}
We have,
$$
\omega\(\prod_{p \in [g^{N-1}, g^N-1]}p_g^*\) \gg  \frac{N^2}{(\log N)^2}.
$$
\end{cor}

\begin{proof} Let us consider only the primes with the first digit equal to 1.
That is primes $p  \in [g^{N-1}, 2g^{N-1}-1]$.
It is enough to consider
$$
W_g(N) = \prod_{p  \in [g^{N-1}, 2g^{N-1}-1]}p_g^*.
$$
and estimate
$$
w_g(N) = \omega\(W_g(N)\).
$$

For a prime $\ell$ we denote by $\nu_g(\ell, N)$ the largest  power $\ell$ that divides
$W_g(N)$.
We have
$$
\nu_g(\ell, N) \le  \sum_{j=1}^\infty R_g(N,\ell^j,0).
$$

Clearly $R_g(N,\ell^j,0) = 0$ if $j\ge  N \log g/\log \ell$.
We also note that for $p  \in [g^{N-1}, 2g^{N-1}-1]$ we have $p_g^* \equiv 1 \pmod g$,
hence only primes $\ell$ with $\gcd(\ell,g)=1$ have to be considered.

Hence, using  Theorem~\ref{thm: RNma} we obtain
\begin{equation}
\label{eq:ell-adic order}
\nu_g(\ell, N) \ll \frac{g^N}{N\ell^{1/2}} .
\end{equation}
By the prime number theorem,
\begin{equation}
\label{eq:W low}
\log W_g(N)  \ge\log \( \(g^{N-1}\)^{(1+o(1))g^{N-1}/(N \log g)} \) =(1+o(1))g^{N-1}.
\end{equation}
On the other hand, using~\eqref{eq:ell-adic order} we obtain
\begin{equation}
\label{eq:W up}
\log W_g(N) = \sum_{\ell \mid W_g(N)} \log \(\ell^{\nu_g(\ell, N)}\)
 \le \frac{g^N}{N} \sum_{\ell \mid W_g(N)}  \frac{\log \ell }{\ell^{1/2}},
\end{equation}
where the sum is  over all primes $\ell \mid  W_g(N)$.
It is easy to see that by the  prime number theorem and partial summation,
for any real $L \ge 3$, we have
$$
 \sum_{\ell  \le L}  \frac{\log \ell }{\ell^{1/2}} \ll L^{1/2}.
$$
Hence
$$
 \sum_{\ell \mid W_g(N)}  \frac{\log \ell }{\ell^{1/2}} \ll \(w_g(N) \log w_g(N)\)^{1/2}
 $$
 and combining this with~\eqref{eq:W low} and~\eqref{eq:W up} we obtain
 $$
  \(w_g(N) \log w_g(N)\)^{1/2}\gg N
$$
 and the result follows.
\end{proof}

One can also derive from Theorem~\ref{thm: RNma}  that
for all but $o(g^N/N)$ primes $p \in [g^{N-1}, g^N-1]$,
the cube-full part of $p_g^*$ is at most
$\psi(N)$ for any function $\psi(N)$ with
$\psi(N) \to \infty$ as $N \to \infty$.
The bounds~\eqref{eq:R triv} give only $\psi(N) N^{1/2}$.

It is also possible to derive other arithmetic applications of
Theorem~\ref{thm: RNma}, for example to show that
$$
 g^N/N\gg \sum_{p \in [g^{N-1}, g^N-1]} \varphi(p_g^*) \gg  g^N/N.
$$

Obtaining better bounds on $R_g(N,m,a)$, in particular, improving those in~\eqref{eq:R triv} and  
Theorem~\ref{thm: RNma}, is also an interesting problem as well, with many potential applications. For example, one can
conjecture that if  $\gcd(g,m)=1$ then we have 
$$
R_g(N,m,a)   \ll  \frac{g^N}{mN}
$$
in a wide range of parameters $m$ and $N$. It is also possible that there is an
asymptotic formula for $R_g(N,m,a)$, but it has to take into account some
local conditions of the same type which are used for
$$(m,g) = (2,2),\, (3,2)
$$
in  Section~\ref{sec:num}.

%\section{Comments, Heuristics  and Numerical Tests}

\section{Heuristics  and Numerical Tests}
\label{sec:num}

Theorem~\ref{thm:PrimeSeq Small g} motivates us to define
$$
\rho_g = \limsup_{N\to \infty} P_g(N)^{1/N}
$$
thus $\rho_g \le \gamma_g$.
We believe that $P_g(N)$ grows exponentially and thus $\rho_{g}>1$  but the
growth is rather slow and thus $\rho_{g}$ is much smaller than $g$. Figure~\ref{fig:png} shows
the growth of $P_g(N)$ for $N\le 50$ and the following estimates seems to be more accurate:
$$
  \rho_2 \approx \rho_3\approx 1.045,\quad \rho_5\approx 1.05,\quad \rho_{10}\approx 2.25.
$$
\begin{figure}[H]
  \centering
  \includegraphics[scale=0.6]{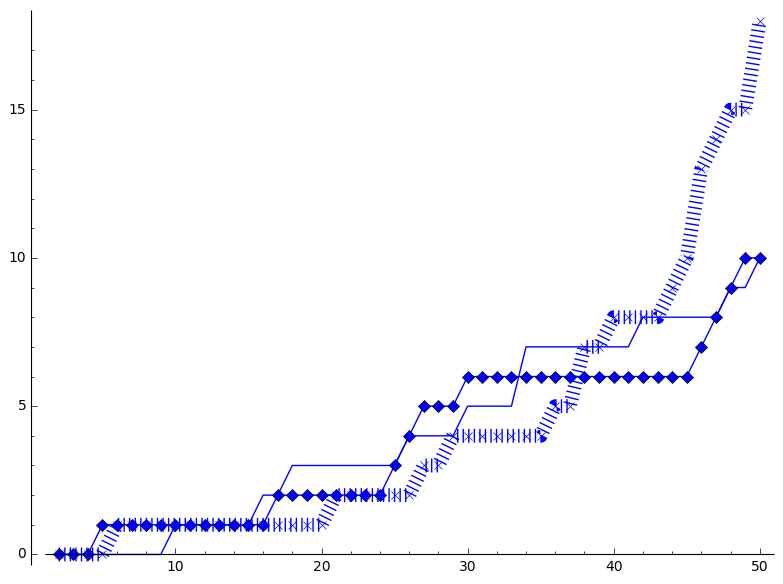}
  \caption{Growth of $P_g(N)$ for $g=2$, $3$ and $5$ in solid, diamonds and dashed lines, respectively.}
  \label{fig:png}
\end{figure}
On the other hand, it seems that the arithmetic structure of $g$ has to be reflected
in any  good approximation
for $\rho_g$. For example, our computation seems to point that  $\rho_{11} \approx 1.6$
and  $\rho_{12} \approx 3.6$.

It is also natural to forbid zero digits and denote by $P_g^*(N)$ the number of prime generating sequences
$\cD = \{d_i\}_{i=0}^{N-1} \in \{1, \ldots, g-1\}^{N-1}$  of
length $N$ that consist only of non-zero digits (in particular the definition is only interesting
for $g \ge 3$). Heuristically, from each value $u_{\cD,g}(N-1)$ contributing to $P_g^*(N-1)$ we seek
through $g-1$ values
for $d_{N-1}  \in \{1, \ldots, g-1\}$ such that  $u_{\cD,g}(N) = d_{N-1}  g^{N-1} + u_{\cD,g}(N-1)$
is prime.
For $N \ge 2$, a naive approximation to the number of  primes
$$
p < g^N \mand p  \equiv u_{\cD,g}(N-1) \pmod {g^{N-1}}
$$
is
$$
\frac{g^N} {\varphi\(g^{N-1}\) \log (g^N)}  =
\frac{g^2} {N \varphi\(g\) \log g}. 
$$
However for $d_{N-1}$  only $g-1$ out $g$ values are admissible.
Hence, we are led to the approximate recursive relation
%\begin{equation}
%\label{eq:P prelim}
$$
P_g^*(N) \approx     \frac{g(g-1)} {N  \varphi(g)  \log g} P_g^*(N-1),
$$
with the inital value $P_g^*(1) = \pi(g-1)$,
%%\end{equation}
which in turn leads us to
the approximation
$$
P_g^*(N) \approx  A_g    \frac{\(g(g-1)\)^{N}} {N!  \(\varphi(g)  \log g\)^N},
$$
where the coefficient $A_g$ depends on the actual values of $P_g^*(N)$ for small values
of $N$.
Using the Stirling formula in the very crude form $N! \approx (N/e)^N$,  we rewrite this as
\begin{equation}
  \label{eq:approximation}
P_g^*(N) \approx    A_g   \( \frac{eg(g-1)} {N \varphi(g)  \log g}\)^N
\end{equation}
for some factor $A_g$ depending only on $g$.
We remark that it is hard to get any explicit formula for $A_g$,
which depends on the initial behaviour of the sequence $P_g^*(N)$.
In particular we see that~\eqref{eq:approximation}, ignoring the presence of the factor $A_g$,
suggests that $P^*_g(N)<1$ (and thus $P^*_g(N) = 0$)
for $N>N_g$ where
\begin{equation}
  \label{eq:Ng}
N_g =\fl{\frac{eg(g-1)} {\varphi(g)  \log g}}.
\end{equation}

Quite naturally, the approximation~\eqref{eq:approximation} is better when $N$ is bigger.
Figure~\ref{fig:PnM} shows the values of the  relative error of the ratio between
$P^*_{g}(N)$ and the  term of the approximation~\eqref{eq:approximation} that varies with $N$, that is,
$$
\alpha_g(N) =  P^*_{g}(N)^{1/N}   \frac{N \varphi(g)  \log g}{eg(g-1)}.
$$
We expect that $\alpha_g(N)$  approximates $A_g^{1/N}$
from below (as the density of primes in the initial intervals $[1,x]$  is  a little higher than $1/\log x$, especially for small values of $x$).   So if the constant $A_g$ 
in~\eqref{eq:approximation} is not too large, it is natural to expect
 that  $\alpha_g(N)$ is close to $1$ in the middle range of $N$ (when the values of  $P^*_{g}(N)$
 are large).  Figure~\ref{fig:PnM} demonstrates the validity of this heuristic prediction.
\begin{figure}[H]
  \centering
  \includegraphics[scale=0.65]{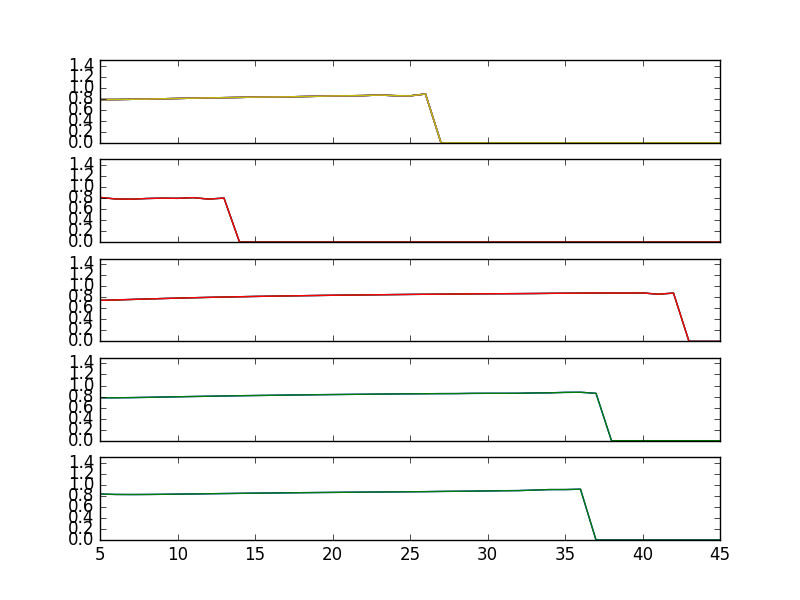}
  \caption{Values of $\alpha_g(N)$ for $g=16,17, 18,20,22$}
  \label{fig:PnM}
\end{figure}
We note that
$$
N_{16} = 29, \qquad N_{17} = 16, \qquad  N_{18} =  47, \qquad N_{20} = 43, \qquad N_{22}= 40,
$$
while Figure~\ref{fig:PnM} shows that the smallest values of $N$ for which  
$\alpha_g(N)=0$ for  $g=16$,  $17$, $18$, $20$ and 
$22$ are $27$, $14$, $43$, $38$ and $37$, respectively.
We also see that in the middle range  of $N$ the function $\alpha_g(N)$
 behaves as almost a constant function, until it suddenly drops to zero.
   So if the constant $A_g$ in~\eqref{eq:approximation} is not too large, it is natural to expect
 that  $\alpha_g(N)$ is close to $1$ in the middle range of $N$ (when the values of  $P^*_{g}(N)$
 are large).
 Certainly, using the
precise values of $N!$ instead of the Stirling formula can also produce  to a more precise
numerical prediction of $P^*_{g}(N)$. For example, we always have $N! > (N/e)^N$ which leads to
slight overestimation of $N_g$  in~\eqref{eq:Ng}.

Another consequence of the approximation~\eqref{eq:approximation} is that one expects $P_g^*(N) = 0$ for a
sufficiently large $N$ (by the Stirling formula,
of size about $e (g-1)/\log g$). This has been tested for $g<40$ using
the computer resources provided by the Santander Supercomputing and,
in all the cases $P_g^*(N) = 0$ for sufficiently large $N$.

We  now make some comments  on the expected growth of $M_g(N)$.
We concentrate on the case of $g=2$.
Clearly a mirror of a prime $p \in [2^{N-1}, 2^N-1]$ is always odd,
which we write as
\begin{equation}
  \label{eq:Cong 2}
 p_2^*  \not \equiv 0 \pmod 2.
\end{equation}
Furthermore considering the digit expansion of  $s \in [2^{N-1}, 2^N-1]$, we derive
\begin{align*}
s  = \sum_{i=0}^{N-1} d_i 2^{i} &\equiv  \sum_{i=0}^{N-1} d_i  (-1)^{i} \\
&  \equiv  (-1)^{N-1} \sum_{i=0}^{N-1} d_i  (-1)^{N-i-1}  \equiv  \pm s_2^* \pmod 3
\end{align*}
we see that for any  prime $p \in [2^{N-1}, 2^N-1]$ we also have
\begin{equation}
  \label{eq:Cong 3}
 p_2^*  \not \equiv 0 \pmod 3.
\end{equation}
The local conditions~\eqref{eq:Cong 2} and~\eqref{eq:Cong 3} (there seems to be no other local conditions)
 coupled with the standard
heuristic suggest that
\begin{equation}
\label{eq:M2N}
\begin{split}
M_2(N) & \approx \(1 -\frac{1}{2}\)^{-1} \cdot  \(1 -\frac{1}{3}\)^{-1} \\
& \qquad \qquad \qquad \cdot \(\pi(2^N) - \pi(2^{N-1})\) \cdot  \frac{\pi(2^N) - \pi(2^{N-1})}{2^{N-1}}\\
& =  6\cdot \frac{\(\pi(2^N) - \pi(2^{N-1})\)^2}{2^{N}}.
 \end{split}
\end{equation}
We have done some computer experiments for $g=2$ and we have plotted the results
of $6\cdot \(\pi(2^N) - \pi(2^{N-1})\)^2/\(2^{N} M_2(N)\)$ in
Figure~\ref{fig:mirror}, which in general seems to be consistent with~\eqref{eq:M2N}. 
However, Figure~\ref{fig:mirror} also indicates that 
there is some small positive bias. 
\begin{figure}[H]
  \centering
  \includegraphics[scale=0.6]{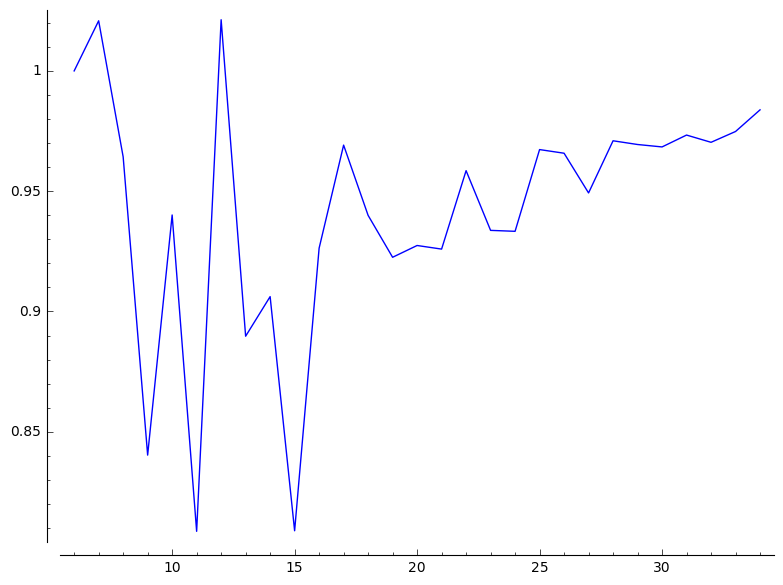}
  \caption{Growth of $6\cdot (\pi(2^N) - \pi(2^{N-1}))^2/(2^{N} M_2(N))$}
%%    in solid line}
  \label{fig:mirror}
\end{figure}

\section*{Acknowledgements}
%
%The author is grateful to Felipe Voloch for  valuable  suggestions used
%Section~\ref{sec:Sing Discr}.

The authors are very grateful to Pieter Moree for introducing him the question
about the mirror primes, and also to Christian Mauduit and Jo{\"e}l Rivat for discussions
of possible approaches to estimating $M_g(N)$.
The authors thankfully acknowledge the computer resources, technical expertise and assistance provided by the the Santander Supercomputing services
at the University of Cantabria.

During the preparation of this paper the first author was partially supported
by project MTM2014-55421-P from the Ministerio de Economia y Competitividad and
the second author was partially supported
by
Australian Research Council Grant~DP140100118.

\end{document}